\documentclass[11pt]{amsart}

\usepackage{amsfonts,latexsym,amsthm,amssymb,graphicx}
\usepackage[all]{xy}
\usepackage{hyperref}

\DeclareFontFamily{OT1}{rsfs}{}
\DeclareFontShape{OT1}{rsfs}{n}{it}{<-> rsfs10}{}
\DeclareMathAlphabet{\mathscr}{OT1}{rsfs}{n}{it}

\CompileMatrices

\newtheorem{theorem}{Theorem}[section]

\newtheorem{prop}[theorem]{Proposition}

{\theoremstyle{definition} \newtheorem{defin}[theorem]{Definition}}
{\theoremstyle{remark} \newtheorem{remark}[theorem]{Remark}
}

\numberwithin{equation}{section}

\newcommand{\Cbb}{{\mathbb{C}}}
\newcommand{\Pbb}{{\mathbb{P}}}
\newcommand{\Qbb}{{\mathbb{Q}}}
\newcommand{\Tbb}{{\mathbb{T}}}
\newcommand{\Zbb}{{\mathbb{Z}}}
\newcommand{\one}{1\hskip-3.5pt1}
\newcommand{\saf}{\,}
\newcommand{\cF}{{\mathcal{F}}}
\newcommand{\cI}{{\mathcal{I}}}
\newcommand{\cL}{{\mathcal{L}}}
\newcommand{\cO}{{\mathcal{O}}}
\newcommand{\cma}{{c_{\mathrm{Ma}}}}
\newcommand{\cstr}{{c_{\mathrm{str}}}}
\newcommand{\qede}{\hfill $\lrcorner$}

\DeclareMathOperator{\rk}{rk}
\DeclareMathOperator{\Eu}{Eu}
\DeclareMathOperator{\codim}{codim}
\DeclareMathOperator{\IC}{IC}
\DeclareMathOperator{\CC}{CC}

\title{
Trim resolutions, stringy and Mather classes, and
IC characteristic cycles}
\author{Paolo Aluffi}
\address{
Mathematics Department, 
Florida State University,
Tallahassee FL 32306, U.S.A.
}
\email{paluffi@fsu.edu}

\date{\today}

\begin{document}

\begin{abstract}
We introduce {\em trim\/} resolutions of complex algebraic varieties, a strengthening 
of the notion of small resolution. We prove that the characteristic cycle of the intersection 
cohomology sheaf of a variety admitting a trim resolution is irreducible and that for 
such varieties the stringy and Chern-Mather classes coincide.
\end{abstract}

\maketitle


\section{Introduction}\label{intro}
Recall that a proper surjective stratified map $\pi: Y\to S$ of complex algebraic varieties
of the same dimension is {\em small\/} if for all proper strata $Z$ of $S$, $2d(Z)< \codim_Z S$, 
where $d(Z)$ denotes the (common) dimension of the fibers of $\pi$ over points of $Z$. 
We say that $\pi$ is a {\em small resolution\/} if $Y$ is nonsingular and $\pi$ is birational 
and small.

We introduce a strengthening of this condition, which we name `trim'
(Definition~\ref{def:contained}). 
Assume $\pi:Y\to S$ is a proper birational map, with $Y$ nonsingular. 
For $d\ge 0$, denote by $Y_d\subseteq Y$ the locally closed subset along which 
the rank of the differential of $\pi$ is~$d$. We say that $\pi$ is {\em trim\/} if 
$\dim Y_d<d$ for all $d<\dim Y$.

Trim maps are small (Proposition~\ref{prop:consma}).
If $\pi:Y\to S$ is small, and for all strata~$Z$ of~$S$ the restriction $\pi^{-1}(Z)\to Z$ is a smooth 
morphism, then $\pi$ is trim (\S\ref{Furrem}). 
For instance, the standard Abel-Jacobi resolution of the theta divisor of a non-hyperelliptic 
curve (which is small, cf.~\cite{MR1642745}) is trim.

Our main result is the following.

\begin{theorem}\label{thm:mainthm}
Let $S$ be a complex algebraic variety admitting a trim resolution. Then
\begin{itemize}
\item
For $s\in S$, the local Euler obstruction $\Eu_S(s)$ equals the Euler characteristic of
the fiber over $s$ in any trim resolution of $S$.
\item
The {\em stringy\/} Chern class of $S$ equals its Chern-Mather class:  
$\cstr(S)=\cma(S)$.
\item
The characteristic cycle of the intersection cohomology sheaf of $S$ is irreducible.
\end{itemize}
\end{theorem}

The definition of characteristic cycle will be recalled below (\S\ref{Lagrapf}).
`Stringy' Chern classes were defined in~\cite{MR2183846, MR2304329} 
(see~\cite{MR2280127} for a lean account). If $\pi: Y\to S$ is a crepant resolution, 
$\cstr(S)$ equals the push-forward of the total Chern class of the tangent bundle
of~$Y$; the equality stated in the theorem is
\begin{equation}\label{eq:strMa}
f_*(c(TY)\cap [Y]) = \cma(S)
\end{equation}
in the Chow group of $S$. In general, stringy Chern classes are defined for
normal varieties with a $\Qbb$-Cartier canonical divisor and at worst log-terminal
singularities, and take value in the Chow group with rational coefficients. These
additional stipulations are not needed for the notion considered in this note; we can 
adopt the left-hand side of~\eqref{eq:strMa} as the definition of stringy Chern class
in the (integral) Chow group of $S$, compatibly with the more general definition.
The fact that the class is independent of the choice of trim resolution is
also a consequence of Theorem~\ref{thm:mainthm}.

We note that Theorem~\ref{thm:mainthm} implies that Batyrev's {\em stringy Euler
number\/} (\cite{MR2001j:14018}) of a variety admitting a trim resolution equals 
its Euler characteristic weighted by the local Euler obstruction.

Theorem~\ref{thm:mainthm} is a consequence of considerations concerning Sabbah's
formalism of {\em conical Lagrangian cycles,\/} to which local Euler obstructions
and Chern-Mather classes relate directly (see~\S\ref{Lagrapf}). Concerning intersection 
cohomology, recall that if $\pi:Y \to S$ is a small resolution, then the 
intersection cohomology sheaf of~$S$ is the push-forward of a shift of the constant 
sheaf on~$Y$ (\cite[\S6.2]{MR696691}). We evaluate the corresponding push-forward 
at the level of characteristic cycles after embedding $S$ in a nonsingular variety~$X$. 
More precisely, we prove (Proposition~\ref{prop:Sabpf}) that, for trim resolutions, the 
Lagrangian push-forward of the zero-section~$T^*_YY$ of the cotangent bundle 
$T^* Y$ equals the conormal cycle $T^*_S X$. The theorem follows from this 
more basic observation, as we show in~\S\ref{thmpf}. 

There is considerable interest in conditions implying that the characteristic cycle of
the intersection cohomology sheaf is irreducible. Lusztig (\cite[13.7, p.~414]{MR1088333}) 
expressed the `hope' that this may be the case for Schubert varieties in flag manifolds of 
type A, D, E. A counterexample was constructed by Kashiwara and Saito for type A in 
$F\ell(8)$ (\cite{MR1458969}, \cite{MR1896039}), while it holds in $F\ell(n)$ for $n\le 7$.
Irreducibility is also known for all Schubert 
varieties in the standard Grassmannian (\cite{MR1084458}), and more generally 
for all Schubert varieties in cominuscule Grassmannians of 
types A, D, and E, see~\cite{MR1451256, MihalceaSingh}.

We hope that Theorem~\ref{thm:mainthm} may help in streamlining such verifications. 
For instance, since the Abel-Jacobi resolution of the theta divisor of a non-hyperelliptic 
curve is trim, the irreducibility of the $\IC$ characteristic cycle in this case (first established 
in \cite{MR1642745}) follows directly from~Theorem~\ref{thm:mainthm}.

The connection between the irreducibility of the $\IC$ characteristic cycle and 
the equality of Chern-Mather and stringy Chern classes was pointed out by B.~Jones
in~\cite[Remark 3.3.2]{MR2628830}, ultimately as an application of the microlocal 
index formula of Dubson and Kashiwara. 
We freely borrow ideas from~\cite{MR2628830} 
in~\S\ref{thmpf}.

A condition on fibers of small resolutions is considered in~\cite{graham}, including a proof
that the condition implies the irreducibility of the IC characteristic cycles.\medskip

{\em Acknowledgments.} This work was supported in part by an award from the Simons
Foundation, SFI-MPS-TSM-00013681. The author also thanks David Massey and
Leonardo Mihalcea for useful conversations, and Caltech for hospitality as most of this 
work was carried out.


\section{Lagrangian push-forward}\label{Lagrapf}
Let $X$ be a nonsingular variety and denote by $T^* X$ the cotangent bundle of $X$.
The {\em conormal variety\/} $T^*_W X$ of a closed subvariety $W\subseteq X$ is the 
closure in $T^* X$ of the conormal variety to the nonsingular part $W^\circ$ of $V$: 
$T^*_W X = \overline{T^*_{W^\circ} X}$. The conormal variety of $X$ itself is the
zero-section $T^*_XX$ of the cotangent bundle. All conormal varieties have dimension
$\dim X$; they determine conormal {\em cycles\/} in $Z_{\dim X} T^* X$.
It will also be convenient to take the projective completion of these constructions:
we will denote by $\Tbb^* X$ the projective completion $\Pbb(T^* X\oplus \one)$ 
of $T^* X$ (here, $\Pbb$ denotes the projective bundle of lines) and by~$\Tbb^*_WX$ 
the closure of the conormal variety in $\Tbb^* X$. 

For a nonsingular variety $X$, we denote by $\cL (X)$ the free abelian group of 
{\em conical Lagrangian cycles\/} in the cotangent bundle $T^* X$.
Conormal cycles are conical Lagrangian, and in fact (cf.~\cite[Lemma~3]{MR1063344}) 
$\cL (X)$ may be realized as the free abelian group on conormal cycles. 
For a closed (and possibly singular) subvariety $V\subseteq X$, we denote by
$\cL(V)$ the subgroup of $\cL(X)$ generated by the conormal cycles $T^*_W X$
with $W\subseteq V$. Thus, elements of~$\cL (V)$ may be viewed as finite integer linear 
combinations $\sum_W m_W T^*_WX$ ranging over closed subvarieties $W$ of $V$.
Clearly $\cL(V)$ is isomorphic to the group of algebraic cycles of~$V$, and in particular
it is independent of the ambient nonsingular variety $X$.

Important invariants of $V$ may be expressed directly in terms of Lagrangian cycles 
by means of intersection-theoretic operations, after taking the projective completion. 
As above, realize $V$ as a closed subvariety of a
nonsingular variety $X$; the results will be independent of the choice of $X$. 
Let $\pi: \Tbb^*_V X\to V$ be the natural projection; and let $\cO(1)$ be the 
tautological line bundle on $\Tbb^* X=\Pbb(T^* X\oplus \one)$.
\begin{itemize}
\item[---]
The {\em local Euler obstruction\/} $\Eu_V: V\to \Zbb$ is
\[
\Eu_V(p)= (-1)^{\dim X-\dim V} \int c(\pi^* TX|_V) c(\cO(1))^{-1}\cap s(\pi^{-1}(p), \Tbb^*_V X)
\]
where $s(\pi^{-1}(p), \Tbb^*_V X)$ denotes the {\em Segre class\/} in the sense 
of~\cite[Chapter~4]{85k:14004};
\item[---]
The {\em Chern-Mather class\/} $\cma(V)\in A_*V$ is
\[
\cma(V)=(-1)^{\dim X-\dim V} c(TX|_V)\cap \pi_* \left(c(\cO(1))^{-1}\cap [\Tbb^*_V X]\right)\saf.
\]
\end{itemize}
Equivalent results were established by C.~Sabbah (\cite[1.2.1, 1.2.2]{MR804052}); we also
address the reader to~\cite{MR1063344} and~\cite{MR2002g:14005} for clear treatments of these
formulas. 

Let $\cF(V)$ denote the abelian group of constructible functions $V\to \Zbb$.
Following~\cite[\S1]{MR2002g:14005}, the relation between conormal cycles and local
Euler obstructions is recorded by the homomorphism
\[
\CC: \cF(V) \to \cL(V)
\]
defined by prescribing
\[
\Eu_W \mapsto (-1)^{\dim W} T^*_W X
\]
for all closed subvarieties $W$ of $V$. 
In fact, $\CC$ is an {\em isomorphism:\/} it simply matches a basis of $\cF(V)$ with a 
basis of $\cL(V)$.

\begin{defin}\label{def:CC}
The {\em characteristic cycle\/} $\CC(\alpha)$ of a constructible function $\alpha$ on $V$ 
is the image of $\alpha$ in $\cL(V)$ under this isomorphism. 
\qede\end{defin}

Likewise, the formula for Chern-Mather classes motivates the introduction of a homomorphism
\[
c_*:\cL(V) \to A_*(V)
\]
defined on generators by
\[
T^*_W X \mapsto (-1)^{\dim W} \cma(W)\saf.
\]
Then the composition $c_*\circ \CC:\cF(V) \to A_*(V)$ agrees with the value at $V$ 
of MacPherson's natural transformation $\cF \leadsto A_*$, where $\cF$ is taken as 
a functor with push-forward defined by Euler characteristics of fibers; 
see~\cite{MR0361141} and~\cite[Example 19.1.7]{85k:14004}. 
For instance, the {\em Chern-Schwartz-MacPherson class\/} of a (possibly singular)
algebraic variety $V$ is the image $c_*(\CC(\one_V))\in A_*V$ of the characteristic 
cycle in~$\cL(V)$ of the constant function $\one_V$.
Sabbah provides an alternative proof of the naturality of this assignment, which is
the main result of~\cite{MR0361141}, by defining a covariant {\em push-forward\/}
\[
\varphi_*: \cL(V') \to \cL(V'')
\]
for every proper map $\varphi: V' \to V''$, making $\cL$ into a functor, in such a way 
that the above homomorphisms define natural transformations
\[
\cF \leadsto \cL\quad, \quad \cL \leadsto A_*
\]
whose composition agrees with MacPherson's natural transformation.

We are interested in explicit formulas for this {\em Lagrangian push-forward.\/}
After embedding~$V''$ in a nonsingular variety $X$ and replacing $V'$ by a resolution $Y$,
we can reduce to the case of a proper morphism $f:Y \to X$ of nonsingular varieties. 
We are specifically interested in the image $f_*(T^*_YY)$ of the zero-section in this
situation. 

Theorem~\ref{thm:mainthm} will be a consequence of the following result. We recall the 
definition of `trim' given in the introduction.

\begin{defin}\label{def:contained}
Let $\pi:Y \to S$ be a proper birational morphism of varieties, with $Y$ nonsingular.
For $0\le d\le \dim Y$, let $Y_d$ denote the locus where the rank of the differential
$d\pi$ equals $d$. Then $\pi$ is a {\em trim resolution\/} if $\dim Y_d<d$
for all $0\le d<\dim Y$.
\qede\end{defin}

\begin{prop}\label{prop:Sabpf}
Let $f: Y\to X$ be a proper morphism of nonsingular varieties, such that $Y \to f(Y)$ is
a trim resolution. Then $f_* (T^*_Y Y)=T^*_{f(Y)} X$.
\end{prop}

\begin{remark}\label{rem:YtoS}
Let $\pi: Y\to S$ be a proper surjective morphism, with $Y$ nonsingular, and 
$\iota:S\hookrightarrow X$ a closed embedding, with $X$ nonsingular; and let 
$f=\iota\circ \pi: Y \to X$ be the composition. 
As explained above, $T^*_{f(Y)} X=T^*_S X$ may be viewed as an element of $\cL(S)$;  
Proposition~\ref{prop:Sabpf} states that if $\pi$ is trim, then this is the image $\pi_*(T^*_YY)$ 
under $\pi_*: \cL(Y)\to \cL(S)$.
\qede\end{remark}

The rest of this section is devoted to the proof of Proposition~\ref{prop:Sabpf}.\smallskip

{\em A priori,\/} $f_*(T^*_Y Y)$ is an integer linear combination of conormal cycles
$\sum_Z m_Z T^*_ZX$, with $Z\subseteq f(Y)$. One of the summands is $T^*_{f(Y)}X$.
The task is to show that if $Y\to f(Y)$ is trim and $Z\subsetneq f(Y)$ is a {\em proper\/} 
subvariety of $f(Y)$, then the coefficient of $T^*_Z X$ in $f_* (T^*_Y Y)$ is~$0$.
Following Sabbah, we write $f$ as the composition
\[
\xymatrix{
Y \ar[r]^-\gamma & Y\times X \ar[r]^-\rho & X
}
\]
where $\gamma$ is the graph of $f$ and $\rho$ is the (smooth) projection.
Let $p: \rho^*(T^* X) \to T^* X$ be the natural mophism. Then  (cf.~\cite[\S2]{MR804052})
the components $T^*_ZX$ appearing in the decomposition of $f_*(T^*_YY)$ are the
irreducible components of the image
\begin{equation}\label{eq:pint}
p \left(\rho^*(T^* X) \cap T^*_{\gamma(Y)}(Y\times X)\right)
\end{equation}
where we view both $\rho^*(T^* X)$ and $T^*_{\gamma(Y)}(Y\times X)$
as subschemes of $T^* (Y\times X)$. 
Let $T^*_ZX$ be one such component, and let $z$ be a general point of $Z$
and $\xi$ a general covector in the fiber $(T^*_ZX)_z$. For $(z,\xi)$ to be in
the image~\eqref{eq:pint}, there must be a point $(y,z)\in \gamma(Y)$ mapping to 
$z$, such that 
$
(0,\xi)\in T^*_y Y\oplus T^*_z X\cong T^*_{(y,z)} (Y\times X)
$
vanishes on $T_{(y,z)}\gamma(Y)$. 

Now, $(y,z)\in \gamma(Y)$ if and only if $y\in f^{-1}(z)$, while $(0,\xi)$ vanishes on 
$T_{(y,z)}\gamma(Y)$ if and only if $\xi$ vanishes on $df(T_yY)$, if and only if
$f^*(\xi)$ vanishes on $T_yY$.
Therefore, $T^*_Z X$ is a component of~\eqref{eq:pint} if and only if for a general
$z\in Z$ and general $\xi\in (T^*_ZX)_z$, the {\em microlocal fiber\/}
(cf.~\cite[Definition~1.2]{MR1084458})
\[
F_{Z,z,\xi} := \text{(zero-scheme of $f^*(\xi)$)}\subseteq f^{-1}(z)
\]
is nonempty. In order to prove Proposition~\ref{prop:Sabpf}, it suffices to prove that for every 
proper subvariety $Z$ of $f(Y)$, the microlocal fiber $F=F_{Z,z,\xi}$ is empty for general 
$z\in Z$ and $\xi\in (T^*_Z X)_z$.

Now let $Z$ be a subvariety of $X$ and let $Z^\circ$ be its nonsingular part. 
Let $W^\circ=f^{-1}(Z^\circ)$ and consider the fiber product
\[
W^\circ \times_{Z^\circ} T^*_{Z^\circ} X\saf,
\]
that is, the pull-back of the conormal bundle $T^*_{Z^\circ} X$ to $W^\circ$. We denote
points of this pull-back by pairs $(w,\xi)$, where $w\in W^\circ$ and 
$\xi\in (T^*_{Z^\circ} X)_{f(w)}$ is a conormal vector to $Z^\circ$ at $f(w)$. 
Note that we have morphisms
\[
\xymatrix{
W^\circ \times_{Z^\circ} T^*_{Z^\circ} X \ar@{^(->}[r] & f^*T^* X|_{W^\circ} 
\ar[r] & T^* Y|_{W^\circ}\saf;
}
\]
we let $\cF$ be the zero-scheme of this composition. Set-theoretically,
\[
\cF=\{(w,\xi)\,\text{s.t.}\, \xi|_{df(T_w Y)} \equiv 0\}\saf.
\]
By construction we have projections $\cF \to W^\circ$, $\cF \to T^*_{Z^\circ} X$.
The microlocal fiber $F$ is naturally identified with the general fiber of the
projection $\cF \to T^*_{Z^\circ} X$.
Therefore, in order to prove Proposition~\ref{prop:Sabpf}, it suffices to show that
if $Z$ is a proper subvariety of $f(Y)$, then $\dim \cF<\dim (T^*_{Z^\circ} X)=\dim X$.

We will evaluate $\dim \cF$ by considering the projection~$\cF \to W^\circ$.

For any $y\in Y$, let $d_y$ be the rank of $df$ at $y$, that is, the dimension 
$\dim(df (T_yY))$ of the image of $T_yY$ in $T_{f(y)}X$. 
For $w\in W^\circ$, $df (T_wY)$ determines the subspace
\[
T_w:=(df (T_wY) + T_{f(w)}Z^\circ)/T_{f(w)}Z^\circ
\]
of the normal space $(TX/TZ^\circ)_{f(w)}$, with dimension $\ge d_w - \dim Z$.
The condition that $\xi\in (T^*_{Z^\circ} X)_{f(w)}=(TX/TZ^\circ)_{f(w)}^*$ 
vanishes along $df (T_wY)$
is equivalent to the condition that it vanishes along~$T_w$; hence
\begin{equation}\label{eq:dimfib}
\dim\{ \xi\in (T^*_{Z^\circ} X)_{f(w)} \,\text{s.t.}\, \xi |_{df(T_w Y)} \equiv 0\}
=(\dim X-\dim Z)-\dim T_w \le \dim X - d_w\saf.
\end{equation}
This is the dimension of the fiber over $w$ of the projection $\cF \to W^\circ$.

Next, for $d\ge 0$ we let
\[
W_d:=Y_d\cap W^\circ = \{w\in W^\circ \,\text{s.t.}\, d_w=d\}
\]
and observe that $W^\circ=\cup_{0\le d\le \dim Y} W_d$. We have the bound
$\dim W_d \le \min (\dim Y_d,\dim W^\circ)$; in particular 
\begin{itemize}
\item
$\dim W_{\dim Y} \le \dim W^\circ < \dim Y$, since $Z$ is assumed to be a proper subvariety
of~$f(Y)$; and
\item
for $0\le d < \dim Y$, we have $\dim W_d \le \dim Y_d<d$ as $Y \to f(Y)$ is assumed to be
trim.
\end{itemize}
Therefore, $\dim W_d < d$ for all $d$. It follows that $\dim \cF < (\dim X-d)+d = \dim X$,
and this concludes the proof of Proposition~\ref{prop:Sabpf}.
\qed


\section{Proof of Theorem~\ref{thm:mainthm}}\label{thmpf}

Let $S$ be a variety admitting a trim resolution $\pi: Y\to S$; let $\iota: S\hookrightarrow X$
be an embedding in a nonsingular variety, and let $f=\iota\circ \pi:Y \to X$ be the composition.
By the covariance of the Lagrangian push-forward and Proposition~\ref{prop:Sabpf},
\[
CC(f_*(\one_Y)) =f_*(CC(\one_Y))= f_*((-1)^{\dim Y}T^*_YY)=(-1)^{\dim S} T^*_SX
=CC(\Eu_S)
\]
where we used the fact that $Y$ is nonsingular (so that $\one_Y=\Eu_Y$) and
that $\dim S=\dim Y$. Since $CC$ is an isomorphism, this implies
\[
f_*(\one_Y) = \Eu_S\saf,
\]
and by definition of push-forward of constructible functions, this means that
\[
\Eu_S(s)=\chi(f^{-1}(s))\saf,
\]
where $\chi$ denotes the Euler characteristic.
Since $f^{-1}(s)=\pi^{-1}(s)$ for $s\in S$, this proves the first assertion of 
Theorem~\ref{thm:mainthm}.

The second assertion follows from the first by applying MacPherson's natural transformation.
 Alternatively, we can use the covariance of $\cL\leadsto A_*$
 (cf.~Remark~\ref{rem:YtoS}):
 \begin{align*}
\cma(S) &=c_*((-1)^{\dim S} T^*_SX)= c_*( \pi_*((-1)^{\dim S} T^*_YY))
= \pi_*(c_*((-1)^{\dim Y}T^*_YY)) \\
&=\pi_*(\cma(Y)) \saf;
\end{align*}
since $Y$ is nonsingular, $\cma(Y)=c(TY)\cap [Y]$, so that 
$\cstr(S):= \pi_*(c(TY)\cap [Y])=\cma(S)$, concluding the proof.

Concerning the third point in Theorem~\ref{thm:mainthm}, recall that the characteristic 
cycle of a complex of sheaves on a variety $S$ embedded in a nonsingular variety $X$ 
is the characteristic cycle (in $\cL(X)$) of its stalk Euler characteristic. (This may be
adopted as the definition of the characteristic cycle, or as an application of the local
index formula, cf.~\cite[Theorem~4.3.25(i)]{MR2050072}.)
As explained in~\S\ref{Lagrapf}, the characteristic cycle may be viewed as an element 
of~$\cL(S)$. The third point of theorem~\ref{thm:mainthm} is a statement about the 
characteristic cycle of the intersection cohomology sheaf $\IC^\bullet_S$ of $S$.

\begin{prop}\label{prop:consma}
Let $\pi: Y \to S$ be a trim proper birational map. Then $\pi$ is small.
\end{prop}

\begin{proof}
Let $\pi: Y\to S$ be a trim proper birational map, let $Z$ be a proper stratum 
of an adapted stratification of $S$, and let $W=\pi^{-1}(Z)$. 
Also, let $d(Z)$ be the common dimension of $\pi^{-1}(z)$ for $z\in Z$;
thus, $\dim W=\dim Z+d(Z)$.  
The differential vanishes along directions tangent to the fibers, therefore 
$\dim \ker d_w\pi \ge d(Z)$ for a general $w$ in a component of maximal dimension 
in $W$. Equivalently, $\rk d_w\pi \le \dim Y-d(Z)$. Since $\pi$~is trim and $Z$ is a 
proper stratum, $\dim W<\rk d_wf$; hence $\dim W< \dim Y-d(Z)$.
Therefore,
\[
\dim Z + d(Z) < \dim Y - d(Z)
\]
i.e., 
\[
2 d(Z) < \dim Y-\dim Z=\dim S-\dim Z = \codim_SZ\saf.
\]
Since this inequality holds for all proper strata $Z$ of $S$, $\pi$ is small.
\end{proof}

Now let $S$ be a variety admitting a trim resolution $\pi: Y\to S$.

By Proposition~\ref{prop:consma}, $\pi$ is small. It follows (cf.~\cite[\S6.2]{MR696691})
that the intersection cohomology sheaf $\IC^\bullet_S$ is the direct image of the intersection 
cohomology sheaf of $Y$. Since $Y$ is nonsingular, the latter is a shift of the constant
sheaf. Therefore,
\[
\IC^\bullet_S = R\pi_* \Qbb_Y[\dim Y]\saf.
\]
(Concerning the shift, we follow the modern convention as in e.g., 
\cite[Remark~4.2.4]{MR2525735}.)
The following formula for the stalk Euler characteristic of $\IC^\bullet_S$ at $z\in S$
is a consequence of standard properties of (derived) direct images:
\[
\chi_z(\IC^\bullet_S) = 
\chi_z(R\pi_* \Qbb_Y[\dim Y]) =
(-1)^{\dim S} \sum_i (-1)^i \dim H^i(\pi^{-1}(z);\Qbb) 
= (-1)^{\dim S}\chi(\pi^{-1}(z))\saf.
\]
By the first assertion in Theorem~\ref{thm:mainthm}, this implies 
\[
\chi_z(\IC^\bullet_S) = (-1)^{\dim S}\Eu_S(z)\saf.
\]
Now embed $S$ in a nonsingular variety $X$.  
The characteristic cycle of $\IC^\bullet_S$ is
\[
CC(\chi_z(\IC^\bullet_S)) = CC((-1)^{\dim S}\Eu_S(z)) = T^*_SX\saf,
\]
and this verifies it is irreducible, completing the proof of 
Theorem~\ref{thm:mainthm}.


\section{Further remarks}\label{Furrem}

1. In this note we have considered the notion of `trim' only for proper birational morphisms  
because that is the case relevant to our application in Theorem~\ref{thm:mainthm}. 
It could be taken as a template for more general proper maps, and it would be interesting
to investigate corresponding generalizations of the main result. For instance, if $f: Y\to X$
is generically finite of degree $m$ onto its image and satisfies the same dimensional 
constraints $\dim Y_d<d$ for $d<\dim Y$, then $f_*(T^*_YY)=m T^*_{f(Y)}X$, with the 
same argument given for Proposition~\ref{prop:Sabpf}.

Similarly, we have restricted attention to {\em complex\/} algebraic varieties to align with some
standard literature, but the results should extend without change to algebraically closed
fields of characteristic $0$; see~\cite{MR1063344} for a treatment of the Lagrangian 
functor in that generality. Also, the results should hold equivariantly; 
see~\cite[\S3.2]{MR4688156} for the relevant equivariant formalism of characteristic
cycles and characteristic classes.
\smallskip

2. The relation between the notion of `trim' and `small' may be clarified by the following
observation.

---A proper birational morphism $\pi: Y\to S$ is small if and only if the fiber product 
$Y\times_S Y$ has a unique component of dimension $\dim Y$ (cf.~e.g., 
\cite[Remark 2.1.2]{MR2067464}). 

---A proper birational morphism $\pi: Y\to S$ is trim if and only if the linear fiber space
associated with the sheaf of differentials $\Omega_{Y|S}$ has a unique component
of dimension $\dim Y$.

Indeed, for all $y\in Y$ we have the exact sequence (tensor~\cite[II.8.11]{MR0463157} by
the residue field $\Cbb(y)$)
\[
\xymatrix{
T^*_{\pi(y)} S \ar[r] & T^*_yY \ar[r] & 
\Omega_{Y|S} \otimes \Cbb(y) \ar[r] & 0
}\saf,
\]
so the `trim' condition is equivalent to the requirement that the {\em co\/}dimension of the locus
where $\dim \Omega_{Y|S} \otimes \Cbb(y)$ equals $d$ is larger than $d$ for all
$d>0$.\smallskip

3. Let $\pi:Y\to S$ be a small stratified map such that the restriction $\pi^{-1}(Z)\to Z$ is a 
smooth morphism for all strata $Z$ of $S$.  
(Here, $\pi^{-1}(Z)$ is the scheme-theoretic inverse image.)
Then $\pi$ is trim.

Indeed, we claim that under this hypothesis, the rank of the differential along $W=\pi^{-1}(Z)$ 
equals $\dim Y-d(Z)$.
To verify this, consider the fiber square
\[
\xymatrix{
W \ar@{^(->}[r]^j \ar[d]_\rho & Y \ar[d]^\pi \\
Z \ar@{^(->}[r]^i & S
}
\]
and let $w\in W$, $z:=\pi(w)\in Z$.
Since the ideal sheaf $\cI_Z$ of $Z$ in $X$ generates the ideal sheaf $\cI_W$ of $W$ in $Y$, 
there is a surjection $\rho^* (\cI_Z/\cI_Z^2) \twoheadrightarrow \cI_W/\cI_W^2$. Chasing the 
diagram
\[
\xymatrix{
 \rho^* (\cI_Z/\cI_Z^2) \ar[r] \ar@{->>}[d] & j^*\pi^* \Omega_X \ar[r] \ar[d] & 
 \rho^* \Omega_Z \ar[r] \ar[d] & 0 \\
 \cI_W/\cI_W^2 \ar[r] & j^* \Omega_Y \ar[r] & \Omega_W \ar[r] & 0 
 }
 \]
shows that the cokernels of the two right-most vertical maps are isomorphic. 
Tensoring by the residue field of $w\in W$ preserves cokernels, and dualizing
shows that the kernels of the vertical maps in
\[
\xymatrix{
T_wW \ar@{^(->}[r] \ar@{->>}[d]_{\rho_*} & T_wY \ar[d]^{\pi_*} \\ 
T_zZ \ar@{^(->}[r] & T_zX
}
\]
are equal. The induced map $\rho_*: T_wW \to T_zZ$ 
is surjective, since $\rho: W\to Z$ is smooth by assumption. 
It follows that $\ker \pi_*=\ker \rho_*$ has dimension $d(Z)$, and
this proves our claim.

The smallness condition for a proper stratum $Z$ is $2 d(Z) < \dim Y-\dim Z$, which is
then equivalent to
\[
\dim W = \dim Z + d(z) < \dim Y- d(Z)=\rk d_w\pi
\]
for all $w\in W$. For all $d<\dim Y$, the locus $Y_d$ in Definition~\ref{def:contained} is a 
(finite) union of inverse images of proper strata, 
so this shows that $\dim Y_d<d$ for $d<\dim Y$ and proves that $\pi$ is trim.



\begin{thebibliography}{dFLNU07}

\bibitem[Alu05]{MR2183846}
Paolo Aluffi.
\newblock Modification systems and integration in their {C}how groups.
\newblock {\em Selecta Math. (N.S.)}, 11(2):155--202, 2005.

\bibitem[Alu07]{MR2280127}
Paolo Aluffi.
\newblock Celestial integration, stringy invariants, and
  {C}hern-{S}chwartz-{M}ac{P}herson classes.
\newblock In {\em Real and complex singularities}, Trends Math., pages 1--13.
  Birkh\"auser, Basel, 2007.

\bibitem[AMSS23]{MR4688156}
Paolo Aluffi, Leonardo~C. Mihalcea, J\"{o}rg Sch\"{u}rmann, and Changjian Su.
\newblock Shadows of characteristic cycles, {V}erma modules, and positivity of
  {C}hern-{S}chwartz-{M}ac{P}herson classes of {S}chubert cells.
\newblock {\em Duke Math. J.}, 172(17):3257--3320, 2023.

\bibitem[Bat99]{MR2001j:14018}
Victor~V. Batyrev.
\newblock Non-{A}rchimedean integrals and stringy {E}uler numbers of
  log-terminal pairs.
\newblock {\em J. Eur. Math. Soc. (JEMS)}, 1(1):5--33, 1999.

\bibitem[BB98]{MR1642745}
P.~Bressler and J.-L. Brylinski.
\newblock On the singularities of theta divisors on {J}acobians.
\newblock {\em J. Algebraic Geom.}, 7(4):781--796, 1998.

\bibitem[BF97]{MR1451256}
Brian~D. Boe and Joseph H.~G. Fu.
\newblock Characteristic cycles in {H}ermitian symmetric spaces.
\newblock {\em Canad. J. Math.}, 49(3):417--467, 1997.

\bibitem[BFL90]{MR1084458}
P.~Bressler, M.~Finkelberg, and V.~Lunts.
\newblock Vanishing cycles on {G}rassmannians.
\newblock {\em Duke Math. J.}, 61(3):763--777, 1990.

\bibitem[Bra02]{MR1896039}
Tom Braden.
\newblock On the reducibility of characteristic varieties.
\newblock {\em Proc. Amer. Math. Soc.}, 130(7):2037--2043, 2002.

\bibitem[dCM04]{MR2067464}
Mark Andrea~A. de~Cataldo and Luca Migliorini.
\newblock The {C}how motive of semismall resolutions.
\newblock {\em Math. Res. Lett.}, 11(2-3):151--170, 2004.

\bibitem[dCM09]{MR2525735}
Mark Andrea~A. de~Cataldo and Luca Migliorini.
\newblock The decomposition theorem, perverse sheaves and the topology of
  algebraic maps.
\newblock {\em Bull. Amer. Math. Soc. (N.S.)}, 46(4):535--633, 2009.

\bibitem[dFLNU07]{MR2304329}
Tommaso de~Fernex, Ernesto Lupercio, Thomas Nevins, and Bernardo Uribe.
\newblock Stringy {C}hern classes of singular varieties.
\newblock {\em Adv. Math.}, 208(2):597--621, 2007.

\bibitem[Dim04]{MR2050072}
Alexandru Dimca.
\newblock {\em Sheaves in topology}.
\newblock Universitext. Springer-Verlag, Berlin, 2004.

\bibitem[Ful84]{85k:14004}
William Fulton.
\newblock {\em Intersection theory}.
\newblock Springer-Verlag, Berlin, 1984.

\bibitem[GJL]{graham}
William Graham, Minyoung Jeon, and Scott~Joseph Larson.
\newblock Irreducible characteristic cycles for orbit closures of a symmetric
  subgroup.

\bibitem[GM83]{MR696691}
Mark Goresky and Robert MacPherson.
\newblock Intersection homology. {II}.
\newblock {\em Invent. Math.}, 72(1):77--129, 1983.

\bibitem[Har77]{MR0463157}
Robin Hartshorne.
\newblock {\em Algebraic geometry}.
\newblock Springer-Verlag, New York, 1977.

\bibitem[Jon10]{MR2628830}
Benjamin~F. Jones.
\newblock Singular {C}hern classes of {S}chubert varieties via small
  resolution.
\newblock {\em Int. Math. Res. Not. IMRN}, (8):1371--1416, 2010.

\bibitem[Ken90]{MR1063344}
Gary Kennedy.
\newblock Mac{P}herson's {C}hern classes of singular algebraic varieties.
\newblock {\em Comm. Algebra}, 18(9):2821--2839, 1990.

\bibitem[KS97]{MR1458969}
Masaki Kashiwara and Yoshihisa Saito.
\newblock Geometric construction of crystal bases.
\newblock {\em Duke Math. J.}, 89(1):9--36, 1997.

\bibitem[Lus91]{MR1088333}
G.~Lusztig.
\newblock Quivers, perverse sheaves, and quantized enveloping algebras.
\newblock {\em J. Amer. Math. Soc.}, 4(2):365--421, 1991.

\bibitem[Mac74]{MR0361141}
R.~D. MacPherson.
\newblock Chern classes for singular algebraic varieties.
\newblock {\em Ann. of Math. (2)}, 100:423--432, 1974.

\bibitem[MS25]{MihalceaSingh}
Leonardo~C. Mihalcea and Rahul Singh.
\newblock Characteristic cycles of {IH} sheaves of simply laced minuscule
  schubert varieties are irreducible.
\newblock {\em Experimental Mathematics}, pages 1--10, 2025.

\bibitem[PP01]{MR2002g:14005}
Adam Parusi{\'n}ski and Piotr Pragacz.
\newblock Characteristic classes of hypersurfaces and characteristic cycles.
\newblock {\em J. Algebraic Geom.}, 10(1):63--79, 2001.

\bibitem[Sab85]{MR804052}
Claude Sabbah.
\newblock Quelques remarques sur la g\'eom\'etrie des espaces conormaux.
\newblock {\em Ast\'erisque}, (130):161--192, 1985.

\end{thebibliography}
\end{document}